\documentclass{amsart}

\usepackage{latexsym,amsmath,amsfonts,amscd,amssymb, amsthm}

\theoremstyle{plain}
\newtheorem{theorem}{Theorem}[section]
\newtheorem{lemma}{Lemma}[section]
\newtheorem{proposition}{Proposition}[section]
\newtheorem{corollary}{Corollary}[section]
 
\theoremstyle{definition}
\newtheorem{definition}{Definition}[section]
\newtheorem{example}{Example}[section]

\theoremstyle{remark}
\newtheorem{remark}{Remark}[section]

\title{Subelliptic Estimates}

\author{David W. Catlin}

\address{Dept. of Mathematics, Purdue Univ., W. Lafayette IN 47906}

\email{catlin@math.purdue.edu}

\author{John P. D'Angelo}

\address{Dept. of Mathematics, Univ. of Illinois, 1409 W. Green St., Urbana IL 61801}

\email{jpda@math.uiuc.edu}

\begin{document}

\maketitle

\section{Introduction}

The purpose of this paper is to clarify some issues concerning subelliptic estimates for the ${\overline \partial}$-Neumann problem
on $(0,1)$ forms. Details of several of the results and examples here do not appear in the literature,
but versions of them have been known to the authors and a few others for a long time, and some have been mentioned
without proof such as in [DK].  Recent interest in this subject helps justify
including them. Furthermore, the situation in two complex dimensions has long been completely understood; one of the main results
there is due to Rothschild and Stein ([RS]) and hence fits nicely into this volume.

First we briefly recall the definition of subelliptic estimate and one consequence of such an estimate. 
See [BS], [C1], [C2], [C3], [DK], [K4], [K5], [KN] for considerable additional discussion. We then discuss the situation
in two complex dimensions, where things are completely understood.
We go on to describe two methods for proving such estimates, Kohn's method of subelliptic multipliers
and Catlin's method of construction of bounded plurisubharmonic functions with large Hessians.

We provide in Proposition 4.4 an example exhibiting the failure of effectiveness for Kohn's algorithm for finding subelliptic multipliers, 
and we give a simplified situation (Theorem 5.1) in which one can understand this algorithm perfectly. This section is taken from [D5].
We go on to discuss some unpublished examples of the first author. These examples provide surprising
but explicit information about how the largest possible value of the parameter $\epsilon$
that arises in a subelliptic estimate is related to the geometry of the boundary.
See Example 7.1 and Theorem 7.2.

Both authors acknowledge discussions with Joe Kohn over the years, and the second author
acknowledges support from NSF Grant DMS-07-53978.

\section{Definition of Subelliptic Estimates}

Let $\Omega$ be a pseudoconvex domain in $\mathbb{C}^{n}$ with smooth boundary, and assume that $p \in b\Omega$. 
Let $T^{1,0}b\Omega$ be the bundle whose sections are $(1,0)$ vectors tangent to $b\Omega$. We may suppose that
there is a neighborhood of $p$ on which $b\Omega$ is given by the vanishing of a smooth function $r$ with $dr(p) \ne 0$.
In coordinates, a vector field $L= \sum_{j=1}^n a_j {\partial \over \partial z_j}$ is a local section of $T^{1,0}b\Omega$ if, on $b\Omega$

$$ \sum_{j=1}^n a_j(z) r_{z_j}(z) = 0. \eqno (1) $$
Then $b\Omega$ is pseudoconvex at $p$ if, whenever (1) holds we have

$$ \sum_{j,k=1}^n r_{z_j {\overline z}_k} (p) a_j(p) {\overline {a_k(p)}} \ge 0. \eqno (2) $$

It is standard to express (2) more invariantly. The bundle $T^{1,0}(b\Omega)$ is a subbundle of $T(b\Omega) \otimes \mathbb{C}$.
The intersection of $T^{1,0}(b\Omega)$ with its complex conjugate bundle is the zero bundle, and their direct sum has fibers
of codimension one in $T(b\Omega) \otimes \mathbb{C}$.
Let $\eta$ be a non-vanishing purely imaginary 1-form that annihilates this direct sum. Then
(1) and (2) together become

$$\lambda(L, {\overline L}) = \langle \eta, [L,{\overline L}] \rangle \ge 0 \eqno (3) $$
on $b\Omega$ for all local sections of $T^{1,0}(b\Omega)$. Formula (3) defines a Hermitian form $\lambda$ on
$T^{1,0}(b\Omega)$ called the Levi form. The Levi form is defined only up to a multiple,
but this ambiguity makes no difference in what we will do. The domain $\Omega$ or its boundary 
$b\Omega$ is called {\it pseudoconvex} if the Levi form is definite everywhere on $b\Omega$; in this case, we multiply by
a constant to ensure that it is {\it nonnegative} definite.
The boundary is {\it strongly pseudoconvex} at $p$ if the Levi form is positive definite there.
Each smoothly bounded domain has an open subset of strongly pseudoconvex boundary points; the point farthest from the origin
must be strongly pseudoconvex, and strong pseudoconvexity is an {\it open condition}.

Subelliptic estimates arise from considering the ${\overline \partial}$-complex on the closed domain ${\overline \Omega}$.
As usual in complex geometry we have notions of smooth differential forms of type $(p,q)$.
We will be concerned only with the case of $(0,1)$ forms here; similar examples and results apply for forms of type $(p,q)$.

A smooth differential $(0,1)$ form $\sum_{j=1}^n \phi_j d{\overline z}^j$, defined near $p$, lies in the domain of ${\overline \partial}^*$
if the vector field $ \sum_{j=1}^n \phi_j {\partial \over \partial z_j}$ lies in $T_z^{1,0}b\Omega$ for $z$ near $p$.
The boundary condition for being in the domain of 
${\overline \partial}^*$ therefore becomes $\sum \phi_j {\partial r \over \partial z_j} = 0$ on the set where $r=0$.
Let $||\psi||$ denote the $L^2$ norm and let $||\psi||_\epsilon$ denote the Sobolev $\epsilon$ norm of $\psi$, where $\psi$ can be either a function or a differential form. The Sobolev norm involves fractional derivatives of order $\epsilon$ of the components of $\psi$.

\begin{definition} A subelliptic estimate holds on $(0,1)$ forms at $p$ if there is a neighborhood $U$ of $p$ and 
positive constants $C$ and $\epsilon$ such that (4) holds for all forms $\phi$, compactly
supported in $U$ and in the domain of ${\overline \partial}^*$.

$$ ||\phi||_\epsilon^2 \le C \left(||{\overline \partial}\phi||^2 + ||{\overline \partial}^* \phi||^2 + ||\phi||^2 \right). \eqno (4) $$
\end{definition}

In this paper we relate the largest possible value of the parameter $\epsilon$ for which (4) holds to the geometry of $b\Omega$.

Perhaps the main interest in subelliptic estimates is the fundamental local regularity theorem of Kohn and Nirenberg [KN].
In the statement of the theorem, the canonical solution to the inhomogeneous Cauchy-Riemann equation is the unique solution
orthogonal to the holomorphic functions.

\begin{theorem} Let $\Omega$ be a smoothly bounded pseudoconvex domain, and assume that there is a subelliptic estimate
at a boundary point $p$. Then there is a neighborhood $U$ of $p$ in ${\overline \Omega}$ with the following property.
Let $\alpha$ be a $(0,1)$ form with $L^2$ coefficients and ${\overline \partial}\alpha = 0$.
Let $u$ be the canonical solution to ${\overline \partial} u = \alpha$. 
Then $u$ is smooth on any open subset of $U$ on which $\alpha$ is smooth. \end{theorem}

It has been known for nearly fifty years ([K1], [K2], [FK])
that there is a subelliptic estimate with $\epsilon = {1 \over 2}$ at each strongly
pseudoconvex boundary point.  One is also interested in {\it global} regularity. See [BS] for a survey of results
on global regularity of the canonical solution. In particular,
on each smoothly bounded pseudoconvex domain, there is a smooth solution
to ${\overline \partial}u = \alpha$ when $\alpha$ is smooth and ${\overline \partial}\alpha =0$, but the canonical solution itself need not be smooth.

\section{subelliptic estimates in two dimensions}

Let $\Omega$ be a pseudoconvex domain in $\mathbb{C}^{2}$ with smooth boundary $M$, and suppose $p \in M$. The statement of Theorem 3.1 below,
resulting by combining the work of several authors,
completely explains the situation.

Assume that $r$ is a defining function for $M$ near $p$. We may choose coordinates
such that $p$ is the origin and 
$$ r(z) = 2{\rm Re}(z_2) + f(z_1,{\rm Im}(z_2)), \eqno (5) $$
where $df(0)=0$. We let ${\bf T}(M,p)$ denote the maximum order of contact of one-dimensional complex analytic curves with $M$ at $p$,
and we let ${\bf T}_{reg}(M,p)$ denote the maximum order of contact of one-dimensional regular complex analytic curves with $M$ at $p$.
We let $t(M,p)$ denote the {\it type} of $M$ at $p$, defined as follows. Let $L$ be a type $(1,0)$ vector field on $M$, with $L(p) \ne 0$.
Then ${\it type}(L,p)$ is the smallest integer $k$ such that there is an iterated bracket
$ {\mathcal L_k} = [...[L_1,L_2],...,L_k]$ for which each $L_j$ is either $L$ or ${\overline L}$ and such that

$$ \langle {\mathcal L_k}, \eta \rangle(p) \ne 0. $$
This number measures the degeneracy of the Levi form at $p$. It 
is independent of the choice of $L$, as $T_p^{1,0}M$ is one-dimensional.
We put $t(M,p) = type(L,p)$.

In two dimensions there is an equivalent method for computing $t(M,p)$.
Consider the Levi form $\lambda(L,{\overline L})$ as a function defined near $p$. We ask how many derivatives one must take 
in either the $L$ or ${\overline L}$ direction to obtain something non-zero at $p$.
Then $c(L,p)$ is defined to be two more than this minimum number of derivatives; we add two because
the Levi form already involves two derivatives. 
In two dimensions it is easy to see that ${\it type}(L,p)=c(L,p)$. This conclusion is false in higher dimensions when the
Levi form has eigenvalues of opposite signs at $p$. It is likely to be true on pseudoconvex domains; see [D1] for more information.

In $\mathbb{C}^{2}$ there are many other ways to compute the type of a point.
The easiest one involves looking at the defining function directly.
With $f$ as in (5), both of these concepts and also both versions of orders of contact mentioned above
equal the order of vanishing of the function $f(z_1,0)$ at the origin. Things are much more subtle and interesting in higher dimensions
regarding these various measurements. See [D1]. Both the geometry and the estimates are easier in $\mathbb{C}^{2}$ than
in higher dimensions; the following theorem explains fully the two dimensional case.

\begin{theorem} Let $\Omega$ be a smoothly bounded pseudoconvex domain in $\mathbb{C}^{2}$, and suppose $p \in b\Omega$. The following are equivalent:

1) There is a subelliptic estimate at $p$ with $\epsilon ={1 \over 2m}$, but for no larger value of $\epsilon$.

2) For $L$ a $(1,0)$ vector field on $b\Omega$ with $L(p) \ne 0$, we have ${\it type}(L,p) =2m$.

3) For $L$ as in 2), we have $c(L,p) = 2m$.

4) There is an even integer $2m$ such that ${\bf T}(b\Omega,p) = 2m$.

5) There is an even integer $2m$ such that ${\bf T}_{reg}(M,p) = 2m$. \end{theorem}

Kohn [K3] established the first subelliptic estimate for domains in $\mathbb{C}^{2}$, assuming that ${\it type}(L,p)$ was finite.
Greiner [Gr] established the converse. To establish the sharp result
that $\epsilon$ could be chosen to be the reciprocal of ${\it type}(L,p)$, Kohn invoked
results of Rothschild-Stein [RS] based on the notion of nilpotent Lie groups. These difficult results establish the equivalence of 1) and 2) above.
Also see for example [CNS] among the many references for estimates in other function spaces for solving the Cauchy-Riemann equations
in two dimensions.

The geometry in two dimensions is easy to understand;
it is quite easy to establish that condition 2) is equivalent to the other conditions from Theorem 3.1, and hence we 
listed all five conditions. In higher dimensions, however, the geometry is completely different. Nonetheless,
based on Theorem 3.1, one naturally seeks a geometric condition for subellipticity in higher dimensions.

\section{subelliptic multipliers}

We next consider the approach of Kohn from [K4] for proving subelliptic estimates.
Let ${\mathcal E}$ denote the ring of germs of smooth functions at $p$.
Recall that $||u||$ denotes the $L^2$-norm of $u$; we use this notation
whether $u$ is a function, a $1$-form, or a $2$-form. We write $||u||_\epsilon$
for the Sobolev $\epsilon$ norm. 

\begin{definition} Assume $f \in {\mathcal E}$. We say that $f$ is a subelliptic multiplier at $p$
if there are positive constants $C$ and $\epsilon$ and a neighborhood $U$ such that

$$ ||f \phi||_\epsilon^2 \le C \left( ||{\overline \partial}\phi||^2 + ||{\overline \partial}^* \phi||^2 + ||\phi||^2  \right) \eqno (6) $$
for all forms $\phi$ supported in $U$ and in the domain of ${\overline \partial}^*$. \end{definition}

\medskip
We will henceforth write $Q(\phi, \phi)$ for 
$||{\overline \partial}\phi||^2 + ||{\overline \partial}^* \phi||^2 + ||\phi||^2 $.
By Definitions 2.1 and 4.1, a subelliptic estimate holds at $p$ if and only if the constant function
$1$ is a subelliptic multiplier at $p$. We recall that when  $b\Omega$ is strongly pseudoconvex
at $p$ we can take $\epsilon = {1 \over 2}$ in (4).

The collection of subelliptic multipliers is a non-trivial ideal in ${\mathcal E}$ closed under taking radicals.
Furthermore, the defining function $r$ and the determinant of the Levi form ${\rm det}(\lambda)$
are subelliptic multipliers. We state these results of Kohn [K4]:

\medskip 
\begin{proposition}. The collection $I$
of subelliptic multipliers is a radical ideal in ${\mathcal E}$; in particular, if $f^N \in I$ for some $N$, then $ f \in I$.
Also, $r$ and ${\rm det}(\lambda)$ are in $I$.
$$||r\phi||_1^2 \le C Q(\phi,\phi) \eqno (7) $$
$$||{\rm det}(\lambda)\phi||_{1 \over 2}^2 \le C Q(\phi, \phi). \eqno (8) $$
\end{proposition}

Kohn's algorithm starts with these two subelliptic multipliers and constructs additional ones. 
We approach the process via the concept of {\it allowable rows}. An $n$-tuple $(f_1,...,f_n)$ of germs of functions is an allowable row 
if there are positive constants $C$ and $\epsilon$ such that, for all $\phi$ as in the definition of subelliptic estimate,

$$ ||\sum_j f_j \phi_j ||_\epsilon ^2 \le C Q(\phi,\phi). \eqno (9) $$

The most important example of allowable row is, for each $j$, the $j$-th row of the Levi form, namely the $n$-tuple
$(r_{z_1 {\overline z_j}},...  ,r_{z_n {\overline z_j}})$.

The following fundamental result of Kohn enables us to pass between allowable rows and subelliptic multipliers:

\begin{proposition} Let $f$ be a subelliptic multiplier such that

$$ ||f \phi||_{2 \epsilon}^2 \le Q(\phi,\phi). \eqno (10) $$
Then the $n$-tuple of functions $({\partial f \over \partial z_1}, ..., {\partial f \over \partial z_1})$ is an allowable row, and we have:

$$ ||\sum_j {\partial f \over \partial z_j} \phi_j ||_\epsilon ^2 \le C Q(\phi,\phi). \eqno (11) $$

Conversely, consider any $n \times n$ matrix $(f_{ij})$ of allowable rows. Then ${\rm det}(f_{ij})$ is a subelliptic multiplier. \end{proposition}

\begin{proof} See [K4] or [D1]. \end{proof}

For domains with real analytic boundary, Kohn's process always terminates in finitely many steps, depending on only the dimension.
Following the process produces two lists of finite length; one of modules of allowable rows, the other of subelliptic multipliers.
The value of the $\epsilon$ obtained from this process depends on both the length of this list and the number of radicals
taken in each step. We will show that there is {\it no positive lower bound} on the value of $\epsilon$ in a subelliptic estimate
obtained from Kohn's process in general. In order to do so we recall some geometric
information and notation from [D1] and [D2]. 

For a real hypersurface $M$ in $\mathbb{C}^{n}$, we recall that
${\bf T}(M,p)$ denotes the maximum order of contact
of one-dimensional complex analytic varieties with $M$ at $p$. We compute this number as follows.
Let $\nu(z)$ denote the order of vanishing operator. Let $z$ be a parametrized holomorphic curve with $z(0)=p$.
We compute the ratio ${\bf T}(M,p,z) = {\nu (z^* r) \over \nu(z)}$ and call it the order of contact of the curve $z$ with $M$ at $p$.
Then ${\bf T}(M,p)$ is the supremum over $z$ of ${\bf T}(M,p,z)$. Later we will generalize this concept.

Next we consider the ring of germs of holomorphic functions ${\mathcal O}$ at $0$ in $\mathbb{C}^{n}$. Some of the ideas also apply to the formal
power series ring; at times we write $R$ or $R_n$ when the statement applies in either setting. See [Cho] for a treatment
of Kohn's algorithm in the formal power series setting.

The maximal ideal in ${\mathcal O}$ is denoted by ${\bf m}$. If $I$ is a proper ideal in ${\mathcal O}$, then the Nullstellensatz
guarantees that its variety ${\bf V}(I)$ is an isolated point if and only if the radical of $I$ equals ${\bf m}$.
In this case the intersection number ${\bf D}(I)$ plays an important role in our discussions.
We put

$$ {\bf D}(I) = {\rm dim}_{\mathbb {C}} {\mathcal O}/I. $$

For such an ideal $I$ we also consider its order of contact ${\bf T}(I)$, defined analogously to the order of contact with a hypersurface.
This number provides a slightly different measurement of the singularity than does ${\bf D}(I)$. See [D1] and [D5] for precise information.

The following proposition is a special case of results from [D2] and [K4]. It gives a simple situation
where one can relate the geometry to the estimates. Note that the geometric conditions 3) through 6)
state in various ways that there is no complex analytic curve in $b\Omega$ through $0$.

\begin{proposition} Let $\Omega$ be a pseudoconvex domain in $\mathbb{C}^{n}$ for which $0 \in b\Omega$, and there are holomorphic functions
$h_j$ such that
the defining equation near $0$ can be written as
$$r(z) = {\rm Re}(z_n) + \sum_{j=1}^N |h_j(z)|^2. \eqno (12) $$
The following are equivalent:

1) There is a subelliptic estimate on $(0,1)$ forms.

2) There is no complex analytic (one-dimensional) curve passing through $0$ and lying in $b\Omega$.

3) ${\bf T}(b\Omega, 0)$ is finite.

4) ${\bf V}(z_n, h_1,...,h_N) = \{0\}$.

5) The radical of the ideal $(z_n, h_1,...,h_N)$ is ${\bf m}$.

6) ${\bf D}(z_n, h_1,...,h_N)$ is finite.
\end{proposition}

Our next example is of the form (12), but it illustrates a new quantitative result.
Let $\Omega$ be a pseudoconvex domain in $\mathbb{C}^{3}$ 
whose defining equation near the origin is given by

$$ r(z) = {\rm Re}(z_3) + |z_1^M|^2 + |z_2^N + z_2 z_1^K|^2. \eqno (13) $$
We assume that $K > M \ge 2$ and $N\ge 3$. 
We note that ${\bf T}(b\Omega,0) = 2 {\rm max}(M,N)$ and that ${\bf D}(z_1^M, z_2^N + z_2 z_1^K, z_3) = MN$.
In the next result we show that Kohn's algorithm for finding subelliptic multipliers gives no lower bound
for $\epsilon$ in terms of the dimension and the type.

\begin{proposition}[Failure of effectiveness] Let $\Omega$ be a pseudoconvex domain whose boundary contains $0$,
and which is defined near $0$ by (13). Then the root taken in the radical required in the second step
of Kohn's algorithm for subelliptic multipliers
is at least $K$, and hence it is independent of the type at $0$. In particular, the procedure in [K4] gives no positive
lower bound for $\epsilon$ in terms of the type. \end{proposition}

\begin{proof} Let $\Omega$ be a domain in $\mathbb{C}^{n+1}$ defined near the origin by (13).
By the discussion in [K4], [D1] or [D5],  Kohn's algorithm reduces to an algorithm in the ring ${\mathcal O}$ in two dimensions.
We therefore write the variables as $(z,w)$ and consider the ideal $(h)$ defined by
$(z^M, w^N + w z^K)$ in two variables. The exponents are
positive integers; we assume $K > M\ge 2$ and $N \ge 3$.  Note that
${\bf D}(h) = MN$ and ${\bf T}(h)= {\rm max}(M,N)$.
 We write $g(z,w)= w^N + w z^K$ and we use subscripts on $g$ to denote partial derivatives.

The algorithm begins with the collection ${\mathcal M}_0$ of {\it allowable rows} spanned by (14)
and the ideal $I_0$ given in (15):

$$ \begin{pmatrix} z^{M-1} & 0 \\
g_z & g_w \end{pmatrix} \eqno (14) $$
There  is only one determinant to take, and therefore 
$$ I_0 = {\rm rad} (z^{M-1} g_w) = (zg_w).\eqno (15) $$
By definition ${\mathcal M}_1$ is the union of ${\mathcal M}_0$ and 
$d(z g_w) = (zg_{wz} + g_w) dz + z g_{ww} dw$. Using the
row notation as before we see that the spanning rows of ${\mathcal M}_1$ are given by (16):
$$  \begin{pmatrix}  z^{M-1} & 0 \\
g_z & g_w \\ zg_{wz} + g_w & z g_{ww} \end{pmatrix}. \eqno (16) $$
It follows that $I_1$ is the radical of the ideal $J_1$ generated by the three possible determinants.

The ideal generated by $zg_w$ and the two new determinants is
$$ J_1 = (zg_w, z^M g_{ww}, zg_z g_{ww} - z g_w g_{zw} - g_w^2). \eqno (17) $$
It is easy to see that 
$$ I_1 = {\rm rad}(J_1) = {\bf m}. \eqno (18) $$
Thus ${\mathcal M}_2$ includes $dz$ and $dw$ and hence $I_2 = (1)$.

The crucial point concerning effectiveness involves the radical
taken in passing from $J_1$ to $I_1$. We prove that we cannot bound this root in terms of $M$ and $N$.
To verify this statement we claim that $z^{K-1}$ is not an element of $J_1$. This claim shows
that the number of roots taken must be at least $K$. Since $K$ can be chosen independently of $M$ and $N$ and also
arbitrarily large, there is no bound on the number of roots taken in terms of the dimension $2$
and the intersection number ${\bf D}(h) = MN$ or the order of contact ${\bf T}(I)={\rm max}(M,N)$.

It remains to prove the claim. If $z^{K-1} \in J_1$, then we could write

$$z^{K-1} = a(z,w) zg_w + b(z,w) z^M g_{ww} + c(z,w) (z g_z g_{ww} - z g_{zw} - g_w^2)
\eqno (19) $$
for some $a,b,c$. We note that $g_{ww}(z,0) = 0$,
that $g_w(z,0) = z^K$, and $g_{zw}(z,0) = Kz^{K-1}$. Using this information
we set $w=0$ in (19) and obtain 

$$z^{K-1} = a(z,0)z^K + b(z,0) 0 + c(z,0) (-z K z^{K-1} + 0). \eqno (20) $$
It follows from (20) that $z^{K-1}$ is divisible by $z^K$; this contradiction proves
that $z^{K-1}$ is not in $J_1$, and hence that passing to $I_1$ requires at least $K$
roots. (It is easy to show, but the information is not needed here, that taking $K$ roots suffices.) 
\end{proof}

This proposition shows that one cannot take radicals
in a controlled fashion unless one revises the algorithm. One might naturally ask whether
we can completely avoid taking radicals. The following example shows otherwise.

\begin{example} Put $n=2$, and let $h$ denote the three functions $(z^2,zw,w^2)$.
Then the three Jacobians obtained are $(z^2, 2w^2,4zw)$. If we tried to use the ideal generated by them, instead of its
radical, then the algorithm would get stuck.
We elaborate; the functions $z^2,zw,w^2$ are not known to be subelliptic multipliers at the start. 
After we compute $I_0$, however, they are known to be subelliptic multipliers and hence we are then allowed to take
the radical. This strange phenomenon (we cannot use these functions at the start, but we can use them after one step)
illustrates one of the subtleties in Kohn's algorithm. \end{example}

\section{Triangular systems}

Two computational difficulties in Kohn's algorithm are
finding determinants and determining radicals of ideals.
We describe a nontrivial class of examples for which finding 
the determinants is easy. At each stage
we require only determinants of triangular matrices. Furthermore we avoid
the computation of uncontrolled radicals; for this class of examples we never
take a root of order larger than the underlying dimension. In order to do so, we deviate from Kohn's algorithm
by treating the modules of $(1,0)$ forms differently. 

We call this class of examples {\it triangular systems}. The author introduced
a version of these examples in [D4], using the term {\it regular coordinate domains},
but the calculations there give a far from optimal value of the parameter
$\epsilon$ in a subelliptic estimate. The version in this section thus improves the work from [D4].
Catlin and Cho [CC] and independently Kranh and Zampieri [KZ] have recently established subelliptic estimates in some specific triangular systems.
The crucial point in this section is that triangular systems enable one to choose
allowable rows in Kohn's algorithm, one at a time and with control on all radicals. In Theorem 5.1 we establish
a decisive result on effectiveness for triangular systems.

\begin{definition}[Triangular Systems]
Let ${\mathcal H}$ be a collection of nonzero elements of
${\bf m} \subset R_n$. We say that ${\mathcal H}$  is a {\it triangular system of full rank} if, possibly after a linear change
of coordinates, there are elements, $h_1,...,h_n \in {\mathcal H}$ 
such that

1) For each $i$ with $1 \le i \le n$,
 we have $ {\partial h_i \over \partial z_j} = 0 $ whenever $j > i$. In other words, $h_i$ depends on only the variables
$z_1,...,z_i$.

2) For each $i$ with $1 \le i \le n$, $h_i(0,z_i) \ne 0$. Here $(0,z_i)$ is the $i$-tuple $(0,...,0,z_i)$. \end{definition}

It follows from 1) that the derivative matrix 
$ dh  = ({\partial h_i \over \partial z_j})$ for $1\le i,j \le n$
is lower triangular. (All the entries above the main diagonal vanish identically.) It follows from 2) that  
$ {\partial h_i \over \partial z_i}(0,z_i) \ne 0$. 
By combining these facts we see that $J={\rm det}(dh)$ is not identically zero. 
Our procedure makes no use of the other elements of ${\mathcal H}$.

Of course any ideal defining a zero-dimensional variety contains a triangular system of full rank.
We are assuming here additionally that the differentials of these functions define the initial module of allowable rows.

\begin{remark} Triangular systems of rank less than $n$
are useful for understanding the generalization of the algorithm where we consider
$q$ by $q$ minors. We do not consider these systems here, and henceforth we drop the phrase
{\it of full rank}, assuming that our triangular systems have full rank. \end{remark}

Let ${\mathcal H}$ be a triangular system.
After renumbering, we may assume that $h_1$ is a function of
$z_1$ alone, $h_2$ is a function of $(z_1,z_2)$, and so on. 
Note that $h_1(z_1) = z_1^{m_1} u_1(z_1)$ for a unit $u_1$, that
$h_2(z_1,z_2) = z_2 u_2(z_2) + z_1 g_2(z_1,z_2)$ for a unit $u_2$, and so on.
After changing coordinates again we may assume that these units are constant.
For example $z_1^{m_1} u_1(z_1) = \zeta_1^{m_1}$, where $\zeta_1$ is a new coordinate.
We may therefore assume that a triangular system
includes functions $h_1,...,h_n$ as follows:

$$ h_1(z) = z_1^{m_1} \eqno (21.1) $$

$$ h_2(z) = z_2^{m_2} + z_1 g_{21}(z_1,z_2) \eqno (21.2) $$

$$ h_3(z) = z_3^{m_3} + z_1 g_{31} (z_1,z_2,z_3) + z_2 g_{32} (z_1,z_2,z_3) \eqno (21.3) $$

$$ h_n(z) = z_n^{m_n} + \sum_{j=1}^{n-1} z_j g_{nj}(z_1,...,z_n). \eqno (21.n) $$
In (21) the holomorphic germs $g_{kl}$ are arbitrary. Our approach works uniformly in them
(Corollary 5.1), but the $\epsilon$ from Kohn's algorithm depends upon them.

Each $h_j$ depends upon only the first $j$ variables and has a pure monomial in $z_j$. 
A useful special case is where each $h_j$ is a
Weierstrass polynomial of degree $m_j$ in $z_j$ whose coefficients depend upon only
the first $j-1$ variables.

\begin{example} Write the variables $(z,w)$ in two dimensions. 
The pair of functions 
$$ h(z,w) = (h_1(z,w),h_2(z,w)) =(z^m, w^n + z g(z,w)), \eqno (22) $$
where $g$ is any element of $R_2$, form a triangular system. \end{example}

\begin{lemma} Let $h_1,...,h_n$ define a triangular system in $R_n$ and let $(h)$
denote the ideal generated by them. Then 

$$ {\bf D}(h) = \prod_{j=1}^n m_j. \eqno (23) $$ \end{lemma}

\begin{proof} There are many possible proofs. One is to compute the vector space
dimension of $R_n/(h)$ by listing a basis of this algebra. The collection
$ \{ z^\alpha \}$ for $ 0 \le \alpha_i \le m_i - 1$ is easily seen to be a basis. 
\end{proof}

We next provide an algorithm that works uniformly over all triangular systems.
The result is a finite list of pairs of subelliptic multipliers; the length of the list 
is the multiplicity from (23). The first pair of multipliers is $(A_1,B_1)$ where both $A_1$ and $B_1$ equal the Jacobian.
The last pair is $(1,1)$. The number of pairs in the list is exactly the multiplicity
(or length) of the ideal $(h)$. The key point is that each $A_j$ is obtained from $B_j$ by taking a controlled root of some
of its factors. In other words, each $B_j$ divides a power of $A_j$, and the power never exceeds the dimension.

We remark that the proof appears at first glance to be inefficient, as delicate machinations 
within it amount to lowering an exponent by one. This inefficiency arises because the proof works 
uniformly over all choices of the $g_{ij}$ in (21).
Perhaps the proof could be rewritten as an induction on the multiplicity.

\begin{theorem} There is an effective algorithm for establishing subelliptic estimates for (domains defined by)
triangular systems. That is, let $h_1,...,h_n$ define a triangular system with $L= {\bf D}(h)= \prod m_j$. The following hold:

1) There is a finite list of pairs of subelliptic multipliers $(B_1,A_1),...., (B_L,A_L)$ such that
$B_1= A_1 = {\rm det} ({\partial h_i \over \partial z_j})$, also $B_L=A_L$, and $B_L$ is a unit.

2) Each $B_j$ divides a power of $A_j$. The power depends on only the dimension $n$ and not on the functions $h_j$. In fact, 

we never require any power larger than $n$. 

3) The length $L$ of the list equals the multiplicity ${\bf D}(h)$ given in (23).\end{theorem}

\begin{proof} The proof is a complicated multiple induction.
For clarity we write out the cases $n=1$ and $n=2$ in full.

When $n=1$
we never need to take radicals. When $n=1$ we may assume $h_1(z_1)= z_1^{m_1}$. 
We set $B_1 = A_1= ({\partial \over \partial z_1})h_1$, and we set
$B_j = A_j = ({\partial \over \partial z_1})^j h_1$. Then $B_1$ is a subelliptic multiplier, and each $B_{j+1}$ is the derivative
of $B_j$ and hence also a subelliptic multiplier; it is the determinant of the one-by-one matrix given by $dB_j$. Since
$h_1$ vanishes to order $m_1$ at the origin, the function $B_{m_1}$ is a non-zero constant. Thus 1) holds.
Here $L=m_1$  and hence 3) holds. Since $B_j=A_j$ we also verify that the power used never exceeds the dimension, and hence 3) holds.
Thus the theorem holds when $n=1$.

We next write out the proof when $n=2$. The initial allowable rows are $dh_1$ and $dh_2$, giving a lower triangular two-by-two matrix,
because ${\partial h_1 \over \partial z_2} = 0$. We set 
$$ B_1 = A_1 = {\rm det} ({\partial h_i \over \partial z_j}) = Dh_1 Dh_2, $$
where we use the following convenient notation:

$$ Dh_k = {\partial h_k \over \partial z_k}. \eqno (24) $$
For $1\le j \le m_2$ we set

$$ B_j = (Dh_1)^2 \ D^j h_2 \eqno (25.1) $$
$$ A_j = Dh_1 \ D^j h_2. \eqno (25.2) $$

Each $B_{j+1}$ is a subelliptic multiplier, obtained by taking the determinant of the allowable matrix
whose first row is $dh_1$ and second row is $dA_j$. Recall that $D^{m_2}h_2$ is a unit.
When $j=m_2$ in (25.2) we therefore find that $A_{m_2}$ is a unit times $Dh_1$. The collection of multipliers is an ideal,
and hence $Dh_1$ is a subelliptic multiplier.
We may use $d(Dh_1)$ as a new allowable first row. Therefore

$$ B_{m_2+1} = D^2(h_1)  Dh_2. $$
Using $d(h_1)$ as the first row and $d(B_{m_2+1})$ as the second row, we obtain
$$ B_{m_2+2} = (D^2 h_1)^2 \ D^2h_2 $$
$$ A_{m_2+2} =  D^2 h_1 \ D^2h_2. $$
Notice again that we took only a square root of the first factor; more precisely, $A_k^2$ is divisible  by $B_k$, where $k=m_2+2$.
Thus each $A_k$ is a multiplier as well. Proceeding in this fashion we obtain
$$ A_{m_2+j} = D^2(h_1)  D^jh_2,$$
and therefore $ A_{2m_2}$ is a unit times $D^2(h_1)$. Thus $d(D^2 h_1)$ is an allowable row.
We increase the index by $m_2$ in order to differentiate $h_1$ once!
Applying this procedure a total of $m_1$ times we see that $B_{m_1 m_2}$ is a unit. 

We started with $A_1=B_1$; otherwise each $B_j$ divides $A_j^2$. Since each $B_j$ is a determinant of a matrix
of allowable rows, each $B_j$ is a subelliptic multiplier. Therefore each $A_j$ is a subelliptic multiplier, and
$A_L=B_L$ is a unit when $L=m_1m_2$. We have verified 1), 2), and 3).

We pause to repeat why we needed to take radicals of order two in the above.
After establishing that $A_j = Dh_1 D^jh_2$ is a multiplier, we use $dA_j$ as an allowable row. 
The next determinant becomes $v= (Dh_1)^2 D^{j+1}h_2$. If we use $v$ as a multiplier,
then we obtain both $Dh_1$ and $D^2 h_1$ as factors. Instead we replace $v$ with $ Dh_1 D^{j+1}h_2$ in order to avoid 
having both $Dh_1$ and $D^2 h_1$ appear.

We now describe this aspect of the process when $n=3$ before sketching the induction. For $n=3$,
we will obtain

$$ A_1= B_1 = (Dh_1) (Dh_2)(Dh_3). $$
After $m_3$ steps we will find that $A_{m_3}$ is a unit times $ Dh_1 \ Dh_2$. To compute the next determinant
we use $dh_1$ as the first row, $d(Dh_1 \ Dh_2)$ as the second row, and $dA_1$ as the third row. Each of these includes $Dh_1$
as a factor, and hence $(Dh_1)^3$ is a factor of the determinant. Hence we need to take a radical of order three.

For general $n$, each matrix of allowable rows used in this process is lower triangular, and hence each determinant taken is a product
of precisely $n$ expressions. As above, the largest number of repeated factors is precisely equal to the dimension.

Now we make the induction hypothesis: we assume that $n\ge 2$, and that $h_1,...,h_n$ defines a triangular system. We assume that
1) and 2) hold for all triangular systems in $n-1$ variables. We set

$$ B_1 = A_1 = {\rm det} ({\partial h_i \over \partial z_j}) = Dh_1 Dh_2 \cdots Dh_n. \eqno (26) $$

We replace the last allowable row by $dA_n$ and take determinants, obtaining

$$ B_2 = Dh_1 Dh_2 \cdots Dh_{n-1} \ Dh_1 Dh_2 \cdots Dh_{n-1}  D^2 h_n \eqno (27) $$
as a subelliptic multiplier. Taking a root of order two, we obtain

$$ A_2 = Dh_1 Dh_2 \cdots Dh_{n-1} D^2h_n \eqno (28) $$
as a subelliptic multiplier. Repeating this process $m_n$ times we obtain 

$$ A_{m_n} = Dh_1 Dh_2 \cdots Dh_{n-1} \eqno (29) $$
as a subelliptic multiplier. We use its differential $dA_{m_n}$ as the $n-1$-st allowable row, and use $dh_n$ as the $n$-th allowable row. 
Taking determinants shows that

$$ A_{m_n+1} = Dh_1 Dh_2 \cdots Dh_{n-2} Dh_1 Dh_2 \cdots D^2h_{n-1} Dh_n \eqno (30) $$
is a subelliptic multiplier.

What we have done? We are in the same situation as before, but we have differentiated the function $h_{n-1}$ one more time, and hence we have taken one step in decreasing the multiplicity of the singularity. We go through the same process $m_{n-1} m_n $ times and we determine that
$A_{m_n m_{n-1}}$ is a subelliptic multiplier which depends upon only the first $n-2$ variables. We then use its differential as the $n-2$-nd allowable row. We obtain, after $m_n m_{n-1} m_{n-2}$ steps, a nonzero subelliptic multiplier independent of the last three variables.
By another induction, after $\prod m_j$ steps, we obtain a unit. Each determinant is the product of $n$ diagonal elements.
At any stage of the process we can have
a fixed derivative of $h_1$ appearing as a factor to at most the first power in each of the diagonal elements. Similarly
a derivative of $Dh_2$ can occur as a factor only in the last $n-1$ diagonal elements. It follows that we never need to take
more than $n$-th root in passing from the $B_k$ (which is a determinant) to the $A_k$. After $L$ steps in all we obtain the unit
$$ D^{m_1}h_1 D^{m_2}h_2...D^{m_n}h_n = A_L = B_L$$ 
as a subelliptic multiplier.
Thus 1), 2), and 3) hold. \end{proof}

\begin{corollary} Let $\Omega$ be a domain defined near $0$ by
$$ {\rm Re}(z_{n+1}) + \sum |h_j(z)|^2, $$
where $h_j$ are as in (21). There is $\epsilon > 0$ such that the subelliptic estimate (4) holds at $0$ 
for all choices of the arbitrary function $g_{jk}$ in (21). 
\end{corollary}

The algorithm used in the proof of Theorem 5.1 differs from Kohn's algorithm.
At each stage we choose a single function $A$ with two properties.
Some power of $A$ is divisible by the determinant of a matrix of allowable rows,
and the differential $dA$ provides a new allowable row. The algorithm takes exactly ${\bf D}(h)$ steps.
Thus we do not consider the modules ${\mathcal M}_k$; instead we add one row at a time to the list of allowable $(1,0)$
forms. By being so explicit we avoid the uncontrolled radicals required in Proposition 4.2.

\begin{remark} The difference in this approach from [K4] can be expressed as follows. We replace the use of uncontrolled
radicals by allowing only $n$-th roots of specific multipliers. On the other hand, we must pay
by taking derivatives more often. The special case when $n=1$ clarifies the difference. \end{remark}

The multiplicity ${\bf D}(h)$ is the dimension over $\mathbb{C}$ of the quotient algebra $R/(h)$. 
This algebra plays an important role in commutative algebra, and it is worth noticing that
the process in Theorem 5.1 seems to be moving through basis elements for this algebra as it finds the $A_j$.
We note however that the multipliers $B_j$ might be in the ideal and hence $0$ in the algebra. We give a simple example.

\begin{example} Let $h(z,w)= (z^2, w^2)$. The multiplicity is $4$. We follow the proof of Theorem 5.1.
We have $(A_1,B_1) = (zw,zw)$. We have $(A_2,B_2) = (z,z^2)$. We have $(A_3,B_3) = (w,w^2)$, and finally
$(A_4,B_4) = (1,1)$. Notice that the $A_j$ give the basis for the quotient algebra, whereas two of the $B_j$
lie in the ideal $(h)$. \end{example}

To close this section we show that we cannot obtain $1$ as a subelliptic multiplier
when the initial set does not define an ${\bf m}$-primary ideal. This result indicates
why the presence of complex analytic curves in the boundary precludes subelliptic estimates on $(0,1)$ forms.
In Theorem 6.2 we state a more precise result from [C1].

\begin{proposition} Let $h_j \in {\bf m}$ for each $j$, and suppose
$(h_1,...,h_K)$ is not ${\bf m}$-primary. Then
the stabilized ideal from the algorithm is not the full ring $R_n$. \end{proposition}

\begin{proof} Since the (analytic or formal) variety defined by the $h_j$ is positive
dimensional, we can find a (convergent or formal) nonconstant
$n$-tuple of power series in one
variable $t$, written $z(t)$,  such that $h_j(z(t))= 0$ in $R_1$ for all $j$.
Differentiating yields

$$ \sum {\partial h_j \over \partial z_k}(z(t)) z_k'(t) = 0.  \eqno (31)$$

Hence the matrix ${\partial h_j \over \partial z_k}$ has a nontrivial kernel, and so
each of its $n$ by $n$ minor determinants $J$ vanishes after substitution of $z(t)$.
Since $J(z(t)) = 0$, 
$$ \sum {\partial J \over \partial z_k}(z(t)) z_k'(t) = 0.  \eqno (32) $$
Hence including the 1-form $dJ$ does not change the collection of vectors annihilated
by a matrix of allowable rows. Continuing we see that
$z'(t)$ lies in the kernel of all new matrices we form from allowable rows, and hence $g(z(t))$ vanishes
for all functions $g$ in the stabilized ideal. Since $z(t)$ is not constant, we conclude
that the variety of the stabilized ideal is positive dimensional,
and hence the stabilized ideal is not $R_n$. \end{proof}

\section{necessary and sufficient conditions for subellipticity}

In the previous sections we have seen a sufficient condition for subellipticity. A subelliptic estimate
holds if and only if the function $1$ is a subelliptic multiplier; there is an algorithmic procedure
to construct subelliptic multipliers beginning with the defining function and the determinant of the Levi form.
Each step of the process decreases the value of $\epsilon$ known to work in (4). If, however, the process terminates
in finitely many steps, then (4) holds for some positive $\epsilon$. Using an important geometric result from [DF],
Kohn [K4] established that the process must terminate when the boundary is real-analytic, and that $1$ is a subelliptic multiplier
if and only if there is no complex variety of positive dimension passing through $p$ and lying in the boundary.

In this section we recall from [C1], [C2], [C3] a different approach to these estimates.
The sufficient condition for an estimate involves the existence of plurisubharmonic functions
with certain properties. Such functions can be used as weight functions in proving $L^2$ estimates. See also [He].
A related approach to the estimates that works on Lipschitz domains appears in [S].

We wish to relate the estimate (4) to the geometry of the boundary.
Let $r$ be a smooth local defining function of a pseudoconvex domain $\Omega$, and assume $0 \in b\Omega$. 
We consider families $\{M_t\}$ of holomorphic curves through $p$
and how these curves contact $b\Omega$ there.  For $ t>0$ we consider nonsingular holomorphic curves $g_t$ as follows:

1) $g_t: \{|\zeta| < t\} \to \mathbb{C}^{n}$ and $g_t(0)=0$.

2) There is a positive constant $c_2$ (independent of $t$) such that, on $\{|\zeta| < 1\}$, we have
$ |g_t'(\zeta)| \le c_2$. 

3) There is a positive constant $c_1$ such that $ c_1 \le |g_t'(0)|$.

We say that the {\it order of contact} of the family $\{M_t\}$ (of holomorphic curves parametrized by $g_t$) 
with $b \Omega$ is $\eta_0$ if $\eta_0$ is the supremum of the set
of real numbers $\eta$ for which

$$ {\rm sup}_\zeta |r(g_t(\zeta))| \le C t^\eta. \eqno (33) $$

The holomorphic curves $g_t$ considered in this definition are all nonsingular. Therefore
this approach differs somewhat from the approach in
[D1] and [D2], where allowing germs of curves with singularities at $0$ is crucial. Our next example provides some insight.

\begin{example} 1) Define $r$ as follows:
$$ r(z) = {\rm Re}(z_3) + |z_1^2 - z_2 z_3|^2 + |z_2|^4. \eqno (34) $$
By [D1] we have ${\bf T}(b\Omega, 0) = 4$. Each curve $\zeta \to g(\zeta)$ whose third component vanishes has contact $4$ at the origin.
On the other hand, consider a nearby boundary point of the form $(0,0,ia)$ for $a$ real.
Then the curve
$$ \zeta \to (\zeta, {\zeta^2 \over ia}, ia)= g_a(\zeta) \eqno (35) $$
has order of contact $8$ at $(0,0,ia)$. By [D2] this jump is the maximum possible; see (39) below for the sharp inequality in general.

2) Following [C1] we jazz up this example by considering 

$$ r(z) = {\rm Re}(z_3) + |z_1^2 - z_2 z_3^l|^2 + |z_2|^4 + |z_1 z_3^m|^2. \eqno (36) $$
for positive integers $l,m$ with $2 \le l \le m$. Again we have ${\bf T}(b\Omega, 0) = 4$. We will construct a family of regular holomorphic curves
$g_t$ with order of contact ${4(2m+l) \over m+2l}$. For $|\zeta| < t$, and $\alpha$ to be chosen, put 

$$ g_t(\zeta) = (\zeta, {\zeta^2 \over (it^\alpha)^l}, i t^\alpha). \eqno (37) $$
Then, pulling back $r$ to $g_t$ we obtain

$$ r(g_t(\zeta)) = {|\zeta|^8 \over |t|^{4 \alpha l}} + |\zeta|^2 |t|^{2 \alpha m}. \eqno (38) $$
Setting the two terms in (38) equal, we obtain
$ |\zeta|^6 = |t|^{4\alpha l + 2 \alpha m}$.
Put $\alpha = {3 \over m + 2 l}$ and then we get $|\zeta|=|t|$.
It follows that
$$ {\rm sup}_\zeta |r(g_t(\zeta))| = 2|t|^\eta, $$
where $\eta = {4 (2m+l) \over m+2l}$. Hence the order of contact of this family is at least $\eta$; in fact it is precisely this value.
Furthermore, by Theorems 6.1 and 6.2 below, there is a subelliptic estimate at $0$ for $\epsilon = {1 \over \eta}$ and this value is
the largest possible. Depending on $l$ and $m$,
the possible values of the upper bound on $\epsilon$ live in the interval $[{1 \over 8}, {1 \over 4}]$.
\end{example}

We next recall the equivalence of subelliptic estimates on $(0,1)$ forms with finite type. 

\begin{theorem} See [C2] and [C3]. Suppose that $b\Omega$ is smooth and pseudoconvex and ${\bf T}(b\Omega,p_0)$ is finite.
Then the subelliptic estimate (4) holds for some $\epsilon > 0$. 
 \end{theorem}

\begin{theorem} See [C1]. Suppose that the subelliptic estimate (4) holds for
some positive $\epsilon$. If $\{M_t\}$ is a family of complex-analytic curves of diameter $t$, then
the order of contact of $\{M_t\}$ with $b \Omega$ is at most ${1 \over \epsilon}$.
 \end{theorem}

In two dimensions, type of a point is an upper semi-continuous function: if the type at $p$ is $t$, then the type is at most $t$ nearby.
In higher dimensions the type of a nearby point can be larger (as well as the same or smaller). 
Sharp local bounds for the type indicate why relating the supremum of possible values
of $\epsilon$ in a subelliptic estimate to the type is difficult in dimension at least $3$.

\begin{theorem} See [D2]. Let $b\Omega$ in ${\mathbb C}^{n}$ be smooth 
and pseudoconvex near $p_0$, and assume ${\bf T}(b\Omega,p_0)$ is finite. Then there
is a neighborhood of $p_0$ on which

$$ {\bf T}(b\Omega,p) \le {{\bf T}(b\Omega,p_0)^{n-1} \over 2^{n-2}}. \eqno (39) $$ 
\end{theorem}

The bound (39) is sharp. When $n=2$ we see that the type at a nearby point can be no larger
than the type at $p_0$. When $n\ge 3$, however, the type can be larger nearby. This failure of upper semi-continuity
of the type shows that the best epsilon in a subelliptic estimate cannot simply be the reciprocal
of the type, as holds in two dimensions. See [D1], [D2], [D3] for more information. Example 7.1 below generalizes Example 6.1.
It is an unpublished result due to the first author.

All these examples are based upon a simple example found by the second author in [D3] to illustrate
the failure of upper semi-continuity of order of contact. See [D1] and [D2] for extensions to higher dimensions
and a proof of (39).

\section{sharp subelliptic estimates}

\begin{example} Consider the local defining function $r$ given by

$$ r(z) = 2{\rm Re}(z_3) + |z_1^{m_1} - f(z_3)z_2|^2 + |z_2^{m_2}|^2
+ |z_2 g(z_3)|^2, \eqno (40) $$
where $m_1$ and $m_2$ are integers at least $2$ and $f$ and $g$ are functions to be chosen. Let $b\Omega$ be the zero set of $r$.
Assume $f(0)=g(0)=0$. It follows by [D1] that
${\bf T}(b\Omega,0)= 2{\rm max}(m_1,m_2)$ and ${\rm mult}(b\Omega,0) = 2 m_1 m_2$.
We will show that we can obtain, for the reciprocal of the largest possible value
of $\epsilon$ in a subelliptic estimate, any value in between these two numbers.

By Theorem 6.1 there is a subelliptic estimate. According to Theorem 6.2, 
to find an upper bound for $\epsilon$  we must find a family $\{M_t\}$ of one-dimensional
complex curves with certain properties. We follow Example 6.1 and define this family $\{M_t\}$ as follows:
$M_t$ is the image of the holomorphic curve

$$ \gamma_t(\zeta) = (\zeta, {\zeta^{m_1} \over f(it)}, it) \eqno (41) $$
on the set where $|\zeta| \le t$.

Pulling back $r$ to this family of curves yields

$$r(\gamma_t(\zeta)) = |{\zeta^{m_1} \over f(it)}|^{2m_2} +  
|{\zeta^{m_1} \over f(it)}|^2 |g(it)|^2. \eqno (42) $$

Reasoning as in Example 6.1, we choose $f$ and $g$ to make the two terms in (42) equal. 
The condition for equality is

$$ |\zeta|^{2m_1m_2 - 2m_1} = |f|^{2m_2-2} |g|^2. \eqno (43) $$

The crucial difference now is that the functions $f$ and $g$, which depend on only one variable, can be chosen as we wish.
In particular, choose a parameter $\lambda \in (0,1]$, and assume that $f$ and $g$ are chosen such that
${\rm log}(|f|) = \lambda {\rm log}(|g|)$. Then (43) gives

$$ (2m_1(m_2-1)){\rm log}(|\zeta|) = (2(m_2-1)\lambda + 2) {\rm log}(|g|). \eqno (44) $$
We obtain from (44) and (45) 

$$ {\rm log} ({|r(\gamma_t(\zeta))| \over 2}) = \left(2m_1 + { (2 - 2 \lambda)2m_1(m_2-1) \over 2(m_2-1)\lambda + 2)}\right)  {\rm log}(|\zeta|). \eqno (45) $$

In order to find a value $\eta$ for which there is a constant $C$ such that $|r|\le C |\zeta|^\eta$, we take logs and see that
we find the ratio ${{\rm log}(|r|) \over {\rm log}(|\zeta|)}$. Using (45) we obtain
the order of contact ${\bf T}$ of this family of curves to be

$$ {\bf T} = 2m_1 + {2(1 -\lambda)m_1 (m_2 -1) \over (m_2-1)\lambda + 1 }. \eqno (46) $$

If in (46) we put $\lambda=1$ then we get ${\bf T} = 2m_1$.
If in (46) we let $\lambda$ tend to $0$, we obtain ${\bf T} = 2m_1+2m_1(m_2-1) = 2m_1m_2$.
\end{example}

In the previous example we may, for example, choose $f(z)= z^p$ and $g(z)=z^q$. If we put $\lambda = {p \over q}$,
then our calculations apply, and (46) is rational. On the other hand, we can achieve the condition ${\rm log}(|f|) = \lambda {\rm log}(|g|)$
by allowing $f$ and $g$ to be functions vanishing to infinite order at $0$ but which are holomorphic in the half plane
${\rm Re}(z_3) < 0$. For example we may define $f$ by
$f(\zeta)= {\rm exp}({-p \over \sqrt{-\zeta}})$ and $g$ the same except that $p$ is replaced by $q$.
By doing so we can allow $\lambda$ in (46) to be real. It is easy to include the limiting value $\lambda=0$, by setting $g=0$.

In order to finish we have to discuss sufficiency. The first author uses the method of weighted $L^2$ estimates.
We let $H(\Phi)$ denote the complex Hessian of a smooth real-valued function $\Phi$. We say that
$H(\Phi) \ge C$ if the minimum eigenvalue of the Hessian is at least $C$ at each point.
One of the crucial steps in the proof of Theorem 6.1 is the following result from [C2],
based upon ideas from [C4]. 

\begin{theorem} Let $\Omega$ be a smoothly bounded domain, defined near a boundary point $p$ by $\{r=0\}$. Suppose
that there is neighborhood $U$ of $p$ such that the following holds:
For each $\delta > 0$, we can find a smooth function $\Phi_\delta$ satisfying

1) $|\Phi_\delta| \le 1$ on $U$. Thus $\Phi_\delta$ is uniformly bounded.

2) $\Phi_\delta$ is plurisubharmonic on $U$. Thus $H(\Phi_\delta) \ge 0$ on $U$.

3) $H(\Phi_\delta) \ge c \delta^{-2 \epsilon}$ on $U \cap \{-\delta < r \le 0\}$. Thus the Hessian of $\Phi_\delta$ blows up
in a precise manner as we approach the boundary.

Then there is a subelliptic estimate of order $\epsilon$ at $p$. \end{theorem}

Using this result it is possible to say more about Example 7.1. One can choose $f$ and $g$ there such that
there is a subelliptic estimate of order $\epsilon$ at the origin, where $\epsilon$ is the reciprocal of the number ${\bf T}$ in (46). 
In particular, for every $\epsilon_0$ in the range $[{1 \over 2m_1m2}, {1 \over 2m_1}]$ there is a domain in $\mathbb{C}^{3}$
such that the largest possible value of $\epsilon$ in a subelliptic estimate is $\epsilon_0$. By changing the function $g$
appropriately, one can create the situation of part 2) of the next result.

\begin{theorem} Let $\epsilon_0$ be in the interval $(0, {1 \over 4}]$.

1) There is a smooth pseudoconvex domain in $\mathbb{C}^{3}$, with defining function (40), such that the subelliptic estimate (4) holds
with $\epsilon$ equal to $\epsilon_0$, but for no larger value of $\epsilon$.
In addition, if $\epsilon_0$ (in the same range)
is rational, then we can choose the domain to be defined by (40), where
$f(z)=z^p$ and $g(z)=z^q$, and hence the defining equation is a polynomial.

2) There is also a smooth pseudoconvex domain in $\mathbb{C}^{3}$, with defining equation (40), such that the
estimate (4) holds for all $\epsilon$ with $0 \le \epsilon < \epsilon_0$, but for
which the estimate fails at $\epsilon_0$.
\end{theorem}

Theorem 7.2 can be extended to higher dimensions. It is much harder
to understand subelliptic estimates on $(0,1)$ forms in three or more dimensions than it is in two dimensions.
The theory for $(0,1)$ forms in two dimensions is analogous to the theory for $(0,n-1)$ forms in $n$ dimensions.
In these cases there is no need to consider the contact with singular varieties, and hence issues involving subelliptic
estimates are controlled by commutators. We conclude by observing that connections between the analysis and the
commutative algebra involved do not reveal themselves in two dimensions, or more generally, when
we consider estimates on $(0,n-1)$ forms. Hence Theorem 3.1 tells only a small part of the full story.

\section*{Bibliography}

[BS] Boas, Harold P. and Straube, Emil J,  Global regularity of the $\overline\partial$-Neumann problem: a survey of the $L^2$-Sobolev theory, 
pages 79-111 in Several Complex Variables, (M. Schneider and Y. T. Siu, eds.), Math. Sci. Res. Inst. Publ. 37, Cambridge Univ. Press, Cambridge, 1999.

[C1] Catlin, D., Necessary conditions for subellipticity
of the $\overline{ \partial }$-Neumann problem, Annals of Math 117(1983), 147-171.

[C2] \underbar{\qquad}, Boundary invariants of pseudoconvex domains, 
Annals of Math. 120(1984), 529-586.

[C3] \underbar{\qquad}, 
Subelliptic estimates for the $\overline{\partial }$-Neumann
problem on pseudoconvex domains, Annals of Math 126(1987), 131-191.

[C4] \underbar{\qquad}, Global regularity for the ${\overline \partial}$-Neumann
problem, pages 39-49 in Complex Analysis of Several Variables (Yum-Tong Siu ed.),
Proc. Symp. Pure Math 41, Amer. Math Soc., Providence, RI, 1984.

[CC] Catlin, D. and Cho, Jae-Seong, Estimates for the ${\overline \partial}$-Neumann problem on regular coordinate domains, preprint.

[CNS] Chang, D. C., Nagel, A., and Stein, E. M., Estimates for the ${\overline
\partial}$-Neumann problem in pseudoconvex domains of finite type in $\mathbb{C}^{2}$.
Acta Math 169(1992), 153-228.

[CS] Chen, So-Chin and Shaw, Mei-chi, Partial differential equations in complex
analysis, Amer. Math Soc./International Press, 2001.

[Cho] Cho, Jae-Seong, An algebraic version of subelliptic multipliers, Michigan Math J., 
Vol. 54  (2006), 411-426. 

[D1] D'Angelo, J. P., Several Complex Variables and the Geometry of Real 
Hypersurfaces, CRC Press, Boca Raton, 1992. 

[D2] \underbar{\qquad},  Real hypersurfaces, orders of contact, and applications, Annals of Math 115(1982), 615-637.

[D3] \underbar{\qquad}, Subelliptic estimates and failure of semi-continuity
for orders of contact, Duke Math. J., Vol. 47(1980), 955-957.

[D4] \underbar{\qquad}, Finite type conditions and subelliptic estimates, Pp. 63-78
in Modern Methods in Complex Analysis, Annals of Math Studies 137, Princeton Univ.
Press, Princeton, 1995.

[D5] \underbar{\qquad} Real and complex geometry meet the Cauchy-Riemann equations, Park City Math Institute Notes, 2008. (to appear)

[DK] D'Angelo, J. and Kohn, J. J., Subelliptic estimates and finite type,
pages 199-232 in Several Complex Variables (M. Schneider and Y. T. Siu, eds.),
Math Sci. Res. Inst. Publ. 37, Cambridge Univ. Press, 1999.

[DF1] Diederich K., and Fornaess, J.E., Pseudoconvex domains with real analytic boundary, 
Annals of Math (2) 107(1978), 371-384.

[FK] Folland, G. B., and  Kohn, J. J., The Neumann problem for the
Cauchy-Riemann complex, {\it Annals of Math. Studies}, 75, Princeton
University Press, 1972.

[Gr] Greiner, P., Subellipticity estimates
of the $\overline{\partial }$-Neumann problem, J. Differential Geometry 9 (1974), 239-260.

[He] Herbig, Anne-Katrin, A sufficient condition for subellipticity of the $\overline\partial$-Neumann operator.  
J. Funct. Anal.  242  (2007),  no. 2, 337--362.

[K1] Kohn, J. J. Harmonic integrals on strongly pseudo-convex manifolds, I,
{\it Annals of Mathematics}, 78, 1963, 112-148.

[K2] \underbar {\qquad} Harmonic integrals on strongly pseudo-convex manifolds, II,
{\it Annals of Mathematics}, 79, 1964, 450-472.

[K3] \underbar {\qquad} Boundary behavior of $\bar{\partial}$ on weakly pseudo-convex
manifolds of dimension two, {\it Journal of Differential Geometry}, 6,
1972, 523-542.

[K4] \underbar {\qquad} Subellipticity of the $\bar{\partial}$-Neumann problem on
pseudo-convex domains: sufficient conditions, {\it Acta Mathematica},
142, March 1979.

[K5] \underbar {\qquad} A Survey of the $\bar{\partial}$-Neumann Problem, {\it Proc. of
Symposia in Pure Math.}, Amer. Math. Soc. 41, 1984, 137-145.

[KN] Kohn, J. J. and Nirenberg, L. Non-coercive boundary value problems, {\it
Comm. Pure Appl. Math.}, 18, 1965, 443-492.

[KZ] Khanh, Tran Vu and Zampieri, G., Precise subelliptic estimates for a class of special domains, (preprint).

[RS] Rothschild, Linda Preiss and Stein, E. M., Hypoelliptic differential operators and nilpotent groups, Acta Math. 137 (1976), 247-320.

[S] Straube, Emil, Plurisubharmonic functions and
subellipticity of the $\overline\partial$-Neumann problem on non-smooth
domains, Math. Res. Lett. 4  (1997), 459-467.

\end{document}